\documentclass{article}
\usepackage{amssymb,amsmath}
\usepackage{graphicx}

\def\qed{\hfill {\hbox{${\vcenter{\vbox{               
   \hrule height 0.4pt\hbox{\vrule width 0.4pt height 6pt
   \kern5pt\vrule width 0.4pt}\hrule height 0.4pt}}}$}}}

\def\tr{\triangleright}

\newtheorem{theorem}{Theorem}
\newtheorem{definition}{Definition}
\newtheorem{lemma}[theorem]{Lemma}

\newtheorem{corollary}[theorem]{Corollary}
\newtheorem{example}{Example}

\newtheorem{remark}{Remark}

\newenvironment{proof}[1][Proof]{\smallskip\noindent{\bf #1.}\quad}%
{\qed\par\medskip}

\date{}

\title{\Large \textbf{N-Degeneracy in rack homology and link invariants}}

\author{Mohamed Elhamdadi\footnote{Email: \texttt{emohamed@math.usf.edu}}
 \and Sam Nelson\footnote{Email: \texttt{knots@esotericka.org}}}

\begin{document}
\maketitle

\begin{abstract}
The aim of this paper is to define a homology theory for racks with 
finite rank $N$ and use it to define invariants of knots
generalizing the CJKLS 2-cocycle invariants related to the invariants
defined in \cite{N1}. For this purpose, we prove that $N$-degenerate 
chains form a sub-complex of the classical complex defining rack homology.
If a rack has rack rank $N=1$ then it is a quandle and our homology theory
coincides with the CKJLS homology theory \cite{CJKLS}. Nontrivial
cocycles are used to define invariants of knots and examples of calculations
for classical knots with up to $8$ crossings and classical links with up 
to $7$ crossings are provided.
\end{abstract}



\section{\large\textbf{Introduction}}

Racks are algebraic structures with axioms derived from Reidemeister moves 
type II and type III. They have been considered by knot theorists in order 
to construct knot and link invariants and their higher analogues (see for 
example \cite{CKS} and references therein). Racks allow a refined and a 
complete algebraic framework in which ones investigates links and 
$3$-manifolds.  They have been studied by many authors and appeared in the 
literature with different names such as automorphic sets and in a special 
case quandles, distributive groupoids, crystals etc. Rack cohomology was 
introduced by Fenn, Rourke and Sanderson \cite{FRS1}.  For each rack $X$ 
and abelian group $A$, they defined cohomology groups $H^n(X,A)$. Since 
then there has been a number of results about this cohomology (see \cite{ EG}, 
\cite{LN} and \cite{FRS2}) with studies from different perspectives. A 
modification of rack cohomology theory led to quandle cohomology theory 
which was developed in \cite{CJKLS} in order to define invariants of classical 
knots and knotted surfaces in state-sum form, called  quandle cocycle 
(knot) invariants.  These invariants can be understood as enhancements of 
the quandle counting invariant, where quandle colorings of a link diagram 
are counted with a weight determined by a quandle cocycle. 

In \cite{CES}, the quandle homology theory was generalized to the case when 
the coefficient groups admit  the structure of Alexander quandles, by 
including  an action of the infinite cyclic group in the boundary operator. 
Using generalizations of quandle homology theory provided by  Andruskiewitsch 
and Gra\~{n}a  \cite{AG}, the quandle cocycle invariants were generalized 
in three different directions in \cite{CEGS}. The second author has studied 
several enhancements of quandle counting invariants of a link $L$ with 
respect to a finite target quandle $T$ (see for example \cite{N1} and 
\cite{AdamN}).  \\

The paper is organized as follows.  After reviewing the basics of racks and 
recalling the notion of rank for racks in section~2,  rack homology is 
considered in section 3.  As our main result, we prove that the degenerate 
chains form a sub-complex of the classical complex defining rack homology
and define a homology theory analogous to quandle 
homology for non-quandle racks which reduces to the usual quandle homology
when the rack is a quandle. In Section 4 we use cocycles in $H^2_{R/ND}$  to 
enhance the rack counting invariant, obtaining a family of link invariants 
which generalize the CJKLS invariants to allow non-quandle racks.  Note that
while a related invariant was defined in \cite{N1}, the fact that 
$N$-degenerate chains form a subcomplex was not proved in \cite{N1}, only 
invariance under blackboard framed isotopy and the $N$-phone cord move. 
In section 5 explicit reduced cocycles are given and used to perform 
computations of the invariants for prime knots and links.  
We end the paper with section 6 in which we make some 
remarks and suggest some open questions for future investigations. 

\section{\large\textbf{Review of Racks}}\label{}

We begin with a definition from \cite{FR}, give examples and then recall from
\cite{N1} the notion of rank for racks.

\begin{definition}
\textup{A \textit{rack} is a set $X$ with two binary operations $\tr$ and
$\tr^{-1}$ satisfying for all $x,y,z\in X$}
\begin{list}{}{}
\item[\textup{(i)}]{$(x\tr y)\tr^{-1} y = x = (x\tr^{-1} y)\tr y$} \textup{and}
\item[\textup{(ii)}]{$(x\tr y)\tr z=(x\tr z)\tr(y\tr z)$}.
\end{list}
\textup{A rack which further satisfies $x\tr x=x$ for all $x\in X$ is 
a \textit{quandle}.}
\end{definition}

\begin{example}\label{ex1}
\textup{Let $V$ be any $\mathbb{Z}_4$-module and define} 
\[\mathbf{u}\tr\mathbf{v}=\mathbf{u}+2\mathbf{v}.\]
\textup{Then $V$ is a rack. The inverse operation
$\mathbf{u}\tr^{-1}\mathbf{v}=\mathbf{u}+2\mathbf{v}$ coincides with the
triangle operation, and we have}
\[
(\mathbf{u}\tr \mathbf{v})\tr \mathbf{w} 
=(\mathbf{u}+ 2\mathbf{v})+2\mathbf{w}
=\mathbf{u}+ 2\mathbf{v}+2\mathbf{w},
\]
\textup{while}
\[
(\mathbf{u}\tr \mathbf{w})\tr (\mathbf{v}\tr \mathbf{w}) 
= (\mathbf{u}+ 2\mathbf{w})+2(\mathbf{v}+2\mathbf{w})
= \mathbf{u}+ 2\mathbf{v}+6\mathbf{w}
= \mathbf{u}+ 2\mathbf{v}+2\mathbf{w}.
\]
\end{example}

\begin{example}
\textup{More generally, let $V$ be any $\ddot{\Lambda}$-module where 
$\ddot{\Lambda}=\mathbb{Z}[t^{\pm 1},s]/(s^2-(1-t)s)$. Then $V$ is a rack
with rack operations}
\[
\mathbf{u}\tr\mathbf{v}=t\mathbf{u}+s\mathbf{v} \quad \mathrm{and}
\quad \mathbf{u}\tr ^{-1}\mathbf{v}=t^{-1}(\mathbf{u}-s\mathbf{v}) 
\]
\textup{since}
\[
(\mathbf{u}\tr \mathbf{v})\tr \mathbf{w} 
=t(t\mathbf{u}+s \mathbf{v})+ s\mathbf{w}
=t^2\mathbf{u}+ts\mathbf{v}+ s\mathbf{w}
\]
\textup{while}
\[
(\mathbf{u}\tr \mathbf{w})\tr (\mathbf{v}\tr \mathbf{w})
=t(t\mathbf{u}+ s\mathbf{w})+ s(t\mathbf{v}+ s\mathbf{w})
=t^2\mathbf{u}+st\mathbf{v}+ (st+s^2)\mathbf{w}
\]
\textup{and since $s^2=(1-t)s$ we have $s=st+s^2$. Racks of this
type are known as \textit{$(t,s)$-racks} in \cite{FR}. Setting 
$s=1-t$ yields a quandle known as an \textit{Alexander quandle}. 
In particular, the rack in example \ref{ex1} is a $(t,s)$-rack with
$t=1$ and $s=2$.}
\end{example}

The rack operations are generally non-associative, with right 
self-distributivity taking the place of associativity. It is a standard 
exercise to show that in any rack $X$ we also have
\begin{list}{$\bullet$}{}
\item{$(x\tr^{-1} y)\tr^{-1} z=(x\tr^{-1} z)\tr^{-1}(y\tr^{-1} z)$,}
\item{$(x\tr y)\tr^{-1} z=(x\tr^{-1} z)\tr(y\tr^{-1} z)$, and}
\item{$(x\tr^{-1} y)\tr z=(x\tr z)\tr^{-1}(y\tr z)$.}
\end{list}

To minimize parentheses, any rack word not containing parentheses will
be associated left-to-right, so that
\[ x_1\tr x_2\tr x_3 \dots \tr x_n = 
(\dots((x_1\tr x_2)\tr x_3)\dots )\tr x_n.\]
In particular, the expression $x\tr^n y$ is an abbreviation for
\[(\dots((x\tr y)\tr y)\dots )\tr y\]
where we have $n$ total $\tr$s.

A similar notion is a \textit{rack power} $x^{\tr k}$, defined recursively
by the rules
\begin{list}{}{}
\item[(i)]{$x^{\tr 1}=x\tr x$} and
\item[(ii)]{$x^{\tr(k+1)}=x^{\tr k}\tr x^{\tr k}$}.
\end{list}
The map $\pi:X\to X$ given by $\pi(x)=x\tr x$ is a bijection
known as a \textit{kink map}; we have $x^{\tr k}=\pi^{k}(x)$.

Given any $x\in X$, we can ask what is the minimal $N\in \mathbb{N}$ such 
that $x^{\tr N} = x$.

\begin{definition}
\textup{Let $X$ be a rack. For each $x\in X$, let $N(x)$ be the smallest
integer $N(x)\in \mathbb{N}$ such that $x^{\tr N} = x$, or $\infty$ if there
is no such $N(x)$. Then the \textit{rack rank} of $X$ is}
\[N(X)=\mathrm{lcm}\{N(x)\ |\ x\in X\}.\]
\textup{If any $N(x)=\infty$ then we have $N(X)=\infty$.}
\end{definition}
\noindent
Note that a quandle is a rack with rack rank $N(X)=1$. We will often 
write $N$ in place of $N(X)$ when the rack in question is understood. 
The rack rank is analogous to the characteristic of a
field or ring; indeed, we might consider the alternative term ``rack 
characteristic.''

For ease of computation, a finite rack $X=\{x_1,\dots, x_n\}$ can be 
represented with a \textit{rack matrix} $M_X$ whose $(i,j)$ entry is 
$k$ where $x_k=x_i\tr x_j$. That is, the rack matrix encodes the
operation table of $(X,\tr)$. Note that the operation table of $(X,\tr^{-1})$
can be recovered from $M_X$, so we do not need to specify both operation
tables to determine a rack structure on $X$. As observed in \cite{N1}, 
every finite rack has finite rack rank equal to the order in $S_n$ of the 
permutation $\pi$ given by the diagonal of the rack matrix. 

\begin{example}\label{example3}
\textup{
Let $X=\mathbb{Z}_4$ with $\mathbf{u}\tr\mathbf{v}=\mathbf{u}+2\mathbf{v}$. 
If we write $\{x_1=1,\ x_2=2, \ x_3=3,\ x_4=4\}$ then the matrix of $X$ is}
\[M_X=\left[\begin{array}{rrrr}
3 & 1 & 3 & 1 \\
4 & 2 & 4 & 2 \\
1 & 3 & 1 & 3 \\
2 & 4 & 2 & 4 \\
\end{array}\right]\]
\textup{Since the diagonal permutation is the transposition $\pi=(13)\in S_4$, 
this rack has rack rank $N=2.$}
\end{example}

Every framed oriented link $L$ has a \textit{fundamental rack} $FR(L)$ 
with generators corresponding to arcs in a blackboard-framed diagram of 
$L$ and relations determined at crossings as pictured:

\[\includegraphics{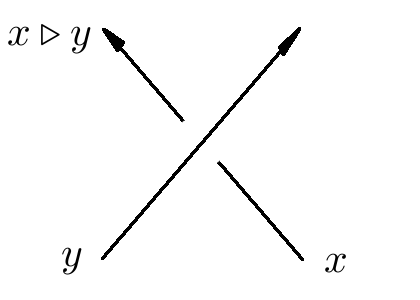} \quad \includegraphics{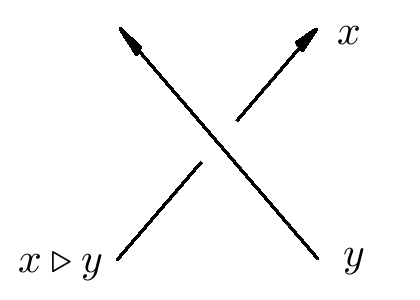}\]

\begin{example}\textup{The blackboard-framed oriented knot below has
listed fundamental rack presentation.}
\[\raisebox{-0.5in}{\includegraphics{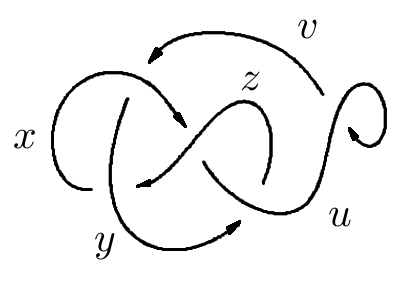}}
 \quad FR(L)=\langle x,y,z,u,v\ | \ u\tr u=v,
y\tr x=v,\ x\tr y= z,\ y\tr u = z,\ x\tr z=u\rangle.\]
\end{example}

The rack axioms encode blackboard-framed isotopy, so that two link diagrams
which are framed isotopic have isomorphic fundamental racks. Indeed, in
\cite{FR} it is shown that the fundamental rack is a complete invariant of
unsplit oriented links in homology 3-spheres.

\[
\raisebox{-0.5in}{\includegraphics{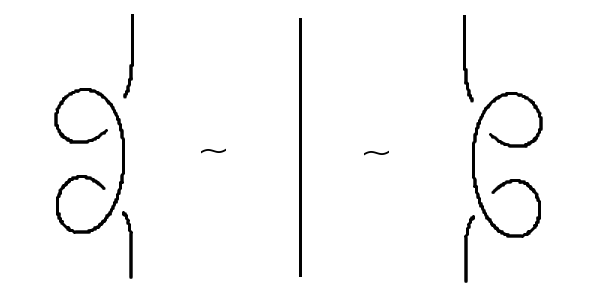}}; \quad 
\raisebox{-0.5in}{\includegraphics{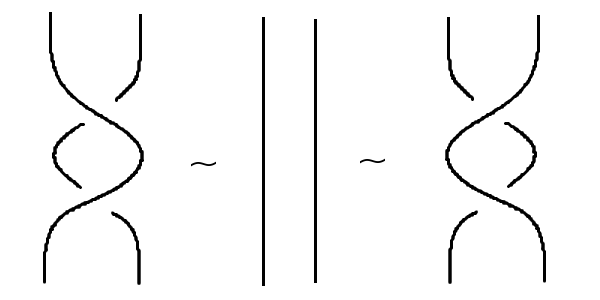}};
\raisebox{-0.5in}{\includegraphics{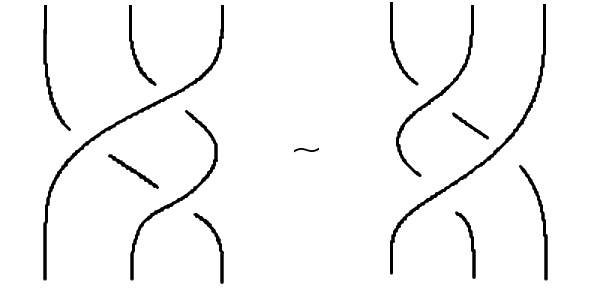}}\]

For any rack $X$, a rack homomorphism $f:FR(L)\to X$ can be visualized as
a \textit{coloring} or labeling of a diagram $L$ by elements of $X$ where
each arc, say $x$, in $L$ gets a label $f(x)\in X$. The set of homomorphisms 
$\mathrm{Hom}(FR(L),X)$
is invariant under blackboard-framed isotopy. By summing these numbers of 
colorings over a complete set of framings modulo $N$, we obtain an invariant
of ambient isotopy (\cite{N1}). 

\begin{theorem}
Let $X$ be a finite rack with rack rank $N$, $L=\cup_{k=1}^cL_k$ a link with 
$c$ components, $W=(\mathbb{Z}_N)^c$, and $FR(L,\mathbf{w})$ the fundamental 
rack of $L$ with framing vector $\mathbf{w}=(w_1,\dots, w_c)\in W$. Then the sum
\[\Phi^{\mathbb{Z}}_X(L)=\sum_{\mathbf{w}\in W} |\mathrm{Hom}(FR(L,\mathbf{w}),X|\]
is an invariant of links, known as the \textit{integral rack counting 
invariant}.
\end{theorem}

The theorem follows from the observation that while colorings of link diagrams
by racks are not preserved under Reidemeister I moves, they are preserved
by the \textit{$N$-phone cord move} where $N$ is the rack rank of $X$:
\[\includegraphics{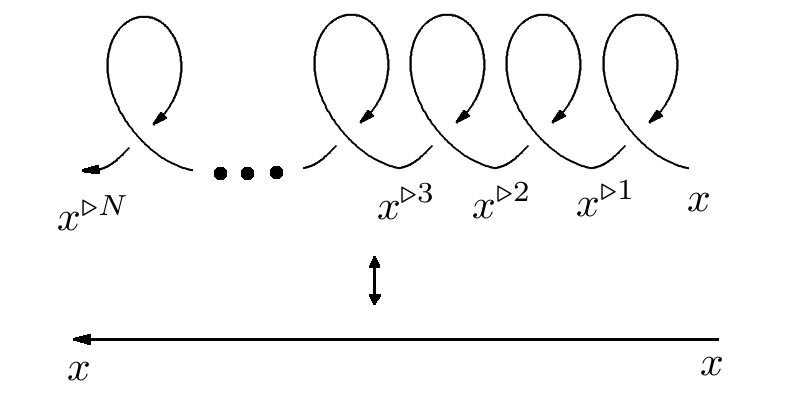}\]
Thus, if two links are equivalent by blackboard-framed moves and $N$-phone cord
moves, their sets of colorings by a rack $X$ with rack rank $N$ are in
one-to-one correspondence.

We can compute the set of rack colorings of a link $L$ by putting the link in
braid form, assigning a generator to each strand at the top and pushing
the colors down the braid; closing the braid, we obtain a system of equations 
with one equation for each strand. We must repeat this computation with $N$ 
stabilization moves on each component to get the total set of colorings.

\[\includegraphics{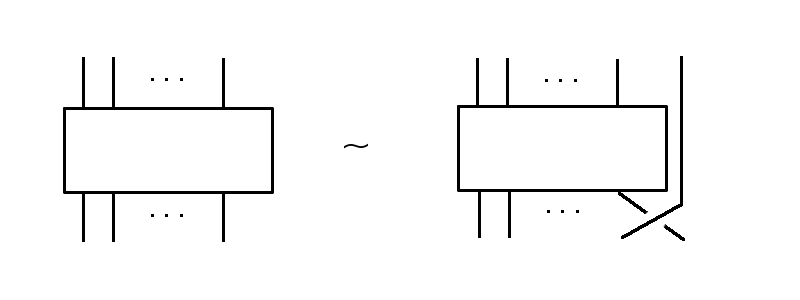}\]

\begin{example}
\textup{The rack in example \ref{example3} has rack rank $N=2$.
The Hopf link has braid presentations with the four possible writhe 
vectors in $W=(\mathbb{Z}_2)^2$ as pictured below; closing each braid
gives us the listed system of equations in $X$. The total number of 
solutions is the integral rack counting invariant $\Phi_X^{\mathbb{Z}}$,
so we have $\Phi_X^{\mathbb{Z}}(L)=4+4+4+8=20$.}
\[\begin{array}{|c|cccc|} \hline
\mathbf{w} &
(0,0) & (1,0) & (0,1) & (1,1) \\ \hline
\raisebox{0.5in}{\textup{Braid Diagram}} &
\includegraphics{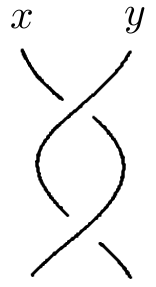} &
\includegraphics{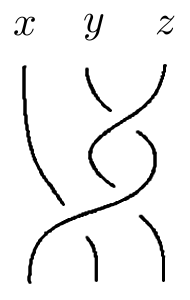} &
\includegraphics{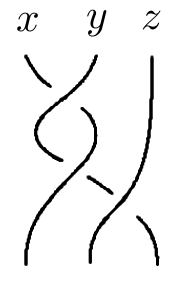} &
\includegraphics{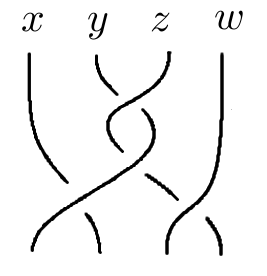} \\ \hline
\mathrm{Equations} &
\begin{array}{rcl}
2x & = & 0\\
2y & = & 0
\end{array} &
\begin{array}{rcl}
2x & = & 0 \\
y  & = & x \\
2z & = & 0 \\
\end{array} &
\begin{array}{rcl}
2x & = & 0 \\
y  & = & z \\
2y & = & 0 \\
\end{array} &
\begin{array}{rcl}
 x & = & 3y \\
 w & = & z \\
2y & = & 2z \\
\end{array} \\ \hline
\# \ \mathrm{solutions} &
4  &
4  &
4  &
8  \\ \hline
\end{array} 
\]
\end{example}

\section{\large\textbf{Rack homology}}

Let $X$ be a finite rack with rack rank $N$ and let $C_n^R(X)$ be 
the free abelian group generated by $n$-tuples $(x_1, \dots, x_n)$ of elements 
of $X$; for $n<1$ set $C_n^R(X)=\{0\}$. Recall that for an $n$-chain, the differential 
$\partial_n:C^R_n(X)\to C^R_{n-1}(X)$ (see \cite{CJKLS} for example) 
is defined on the generator $(x_1,\dots,x_n)$ by
\begin{eqnarray*}
\partial_{n}(x_1, x_2, \dots, x_n) & = &
\sum_{i=2}^{n} (-1)^{i}\left[ (x_1, x_2, \dots, x_{i-1}, x_{i+1},\dots, x_n)\right. \\
& & \quad \quad - \left. (x_1 \tr x_i, x_2 \tr x_i, \dots, x_{i-1}\tr x_i, x_{i+1}, \dots, x_n) \right]
\end{eqnarray*}
for $n \geq 2$ and $\partial_n=0$ for $n \leq 1$ and extended to 
$C^R_n(X)$ by linearity.

\begin{definition}\label{Ndegenerate}
\textup{Let $X$ be a finite rack of rack rank $N$ and let 
$n\geq 2$ be an integer.  A rack $n$-chain $\mathbf{x}\in C^R_n(X)$ is 
\textit{$N$-degenerate}
if $\mathbf{x}$ is a linear combination of chains of the form}
\[\sum_{k=1}^N 
(x_1,\dots, x_{i-1},x_i^{\tr k}, x_i^{\tr (k+1)}, x_{i+2}, \dots, x_n).\]
\end{definition}

We now come to our main result:

\begin{theorem}
Let $X$ be a finite rack with rack rank $N$. Then the $N$-degenerate $n$-chains
form a subcomplex of the complex $C_*^R(X)$.
\end{theorem}

\begin{proof} Let 
\[\mathbf{x}=\sum_{k=1}^N 
(x_1,\dots, x_{i-1},x_i^{\tr k}, x_i^{\tr (k+1)}, x_{i+2}, \dots, x_n).\]
One checks that the $(n-1)$-chain $\partial_n\mathbf{x}$ is $N$-degenerate:

\begin{eqnarray}\label{EQ2}
\partial_n \mathbf{x} & = & 
\partial_{n}\left[
\sum_{k=1}^N 
(x_1,\dots, x_{i-1},x_i^{\tr k}, x_i^{\tr (k+1)}, x_{i+2}, \dots, x_n)\right] 
\nonumber \\
& = &\sum_{k=1}^N  \partial_{n}
(x_1,\dots, x_{i-1},x_i^{\tr k}, x_i^{\tr (k+1)}, x_{i+2}, \dots, x_n) \nonumber \\& = &  
\sum_{k=1}^N \left\{ \sum_{j=2}^{i-1} (-1)^{j} [ (x_1,\dots,x_{j-1}, 
x_{j+1},\dots, x_i^{\tr k}, x_i^{\tr (k+1)}, x_{i+2}, \dots, x_n) \right.
\nonumber \\ 
& & \quad \quad
\left. - (x_1 \tr x_j,\dots, x_{j-1} \tr x_j, x_{j+1},\dots, x_i^{\tr k}, 
x_i^{\tr (k+1)}, x_{i+2}, \dots, x_n) ] \right\} \nonumber \\
& & +
\sum_{k=1}^N 
\left\{ 
(-1)^i\left[ (x_1,\dots, x_{i-1}, x_i^{\tr (k+1)}, x_{i+2}, \dots, x_n) 
\right. \right. \nonumber \\ & &
\quad \quad - \left. \left.
(x_1 \tr x_i^{\tr k} ,\dots, x_{i-1} \tr x_i^{\tr k}, x_i^{\tr (k+1)}, x_{i+2}, \dots, x_n)\right]  \right.
\nonumber \\ 
&  & \quad\quad +
(-1)^{i+1}\left[ (x_1,\dots, x_{i-1}, x_i^{\tr k}, x_{i+2}, \dots, x_n)
\right. \nonumber \\
& & \quad\quad - 
\left. \left.(x_1 \tr x_i^{\tr (k+1)} ,\dots, x_{i-1} \tr x_i^{\tr (k+1)}, x_{i}^{\tr k} \tr x_i^{\tr (k+1)} ,x_{i+2}, \dots, x_n)\right]  \right\}
\nonumber \\ 
& & +
\sum_{k=1}^N \left\{\sum_{j=i+2}^{n} (-1)^{j} 
\left[ (x_1,\dots, x_{i-1},x_i^{\tr k}, x_i^{\tr (k+1)}, \dots, x_{j-1}, x_{j+1},\dots, x_n)  \right. \right.
\nonumber \\
&  & \quad \quad \left.\left.
 - (x_1 \tr x_j,\dots, x_{i-1} \tr x_j,x_i^{\tr k} \tr x_j, x_i^{\tr (k+1)} \tr x_j, \dots, x_{j-1} \tr x_j,x_{j+1}, \dots, x_n) \right]
\right\}.
\end{eqnarray}
One observes that the following sum vanishes:
\[
\sum_{k=1}^N \left\{ \left[ (x_1,\dots, x_{i-1}, x_i^{\tr (k+1)}, x_{i+2}, \dots, x_n) - (x_1 \tr x_i^{\tr k} ,\dots, x_{i-1} \tr x_i^{\tr k}, x_i^{\tr (k+1)}, x_{i+2}, \dots, x_n)\right] -\right.\]
\begin{equation}
\left.\left[ (x_1,\dots, x_{i-1}, x_i^{\tr k}, x_{i+2}, \dots, x_n) - (x_1 \tr x_i^{\tr (k+1)} ,\dots, x_{i-1} \tr x_i^{\tr (k+1)}, x_i^{\tr k}\tr x_i^{\tr (k+1)},x_{i+2}, \dots, x_n)\right] \right\} \\ 
\end{equation}
because $x^{\tr N}=x$ and thus
\[ \sum_{k=1}^N \left[ (x_1,\dots, x_{i-1}, x_i^{\tr (k+1)}, x_{i+2}, \dots, x_n)- (x_1,\dots, x_{i-1}, x_i^{\tr k}, x_{i+2}, \dots, x_n)\right]=0
\]
and
\[
(x_1 \tr x_i^{\tr k} ,\dots, x_{i-1} \tr x_i^{\tr k}, x_i^{\tr (k+1)}, x_{i+2}, \dots, x_n) =  (x_1 \tr x_i^{\tr (k+1)} ,\dots, x_{i-1} \tr x_i^{\tr (k+1)}, x_i^{\tr k}\tr x_i^{\tr (k+1)},x_{i+2}, \dots, x_n).
\]
This last difference is zero since by [15, Lemma 1] we have 
$\forall u \in X,\; u \tr x_i^{\tr (k+1)}= u \tr (x_i^{\tr k} \tr x_i^{\tr k})= u \tr x_i^{\tr k}$ and similarly $x_i^{\tr k} \tr x_i^{\tr (k+1)}= x_i^{\tr k} \tr x_i^{\tr k}$.

Using self-distributivity, the last term in  equation (\ref{EQ2}) can be re-written as (so it makes the last sum fits in the definition of degenerate chains)
\[
(x_1 \tr x_j,\dots, x_{i-1} \tr x_j,x_i^{\tr k} \tr x_j,( x_i^{\tr k} \tr x_j)\tr (x_i^{\tr k} \tr x_j ), \dots, x_{j-1} \tr x_j,x_{j+1}, \dots, x_n) 
\]
\end{proof}

Recall from \cite{CJKLS} that $C_\ast^{\rm R}(X)
= \{C_n^{\rm R}(X), \partial_n \}$ is a chain complex.
Let $C_n^{\rm ND}(X)$ be the subset of $C_n^{\rm R}(X)$ generated
by the $N$-degenerate $n$-chains as in Definition \ref{Ndegenerate}
if $n \geq 2$;
otherwise let $C_n^{\rm ND}(X)=0$. If $X$ is a rack, then
$\partial_n(C_n^{\rm ND}(X)) \subset C_{n-1}^{\rm ND}(X)$ and
$C_\ast^{\rm ND}(X) = \{ C_n^{\rm ND}(X), \partial_n \}$ is a sub-complex of
$C_\ast^{\rm
R}(X)$. Put $C_n^{\rm R/ND}(X) = C_n^{\rm R}(X)/ C_n^{\rm ND}(X)$ and 
$C_\ast^{\rm R/ND}(X) = \{ C_n^{\rm R/ND}(X), \partial'_n \}$,
where $\partial'_n$ is the induced homomorphism.
Henceforth, all boundary maps will be denoted by $\partial_n$.

For an abelian group $G$, define the chain and cochain complexes
\begin{eqnarray}
C_\ast^{\rm W}(X;G) = C_\ast^{\rm W}(X) \otimes G, \quad && \partial =
\partial \otimes {\rm id}; \\ C^\ast_{\rm W}(X;G) = {\rm Hom}(C_\ast^{\rm
W}(X), G), \quad
&& \delta= {\rm Hom}(\partial, {\rm id})
\end{eqnarray}
in the usual way, where ${\rm W}$ 
 $={\rm ND}$, ${\rm R}$, ${\rm R/ND}$.

\begin{definition}\textup{
The $n$\/th {\it reduced rack homology group\/}  and the $n$\/th
{\it reduced rack cohomology group\/ }  of a rack $X$ with coefficient group 
$G$ are}
\begin{eqnarray}
H_n^{R/ND}(X; G) 
 = H_{n}(C_\ast^{R/ND}(X;G)), \quad
H^n_{R/ND}(X; G) 
 = H^{n}(C^\ast_{R/ND}(X;G)). \end{eqnarray}
\end{definition}

As an example, let us now compute the first and second homology groups of the $(t,s)$-rack 
$X=\mathbb{Z}_4$ with $t=1$ and $s=2$ as in example \ref{example3}.  We have the following

\begin{lemma}  Let $X=\mathbb{Z}_4$ with $\mathbf{u}\tr\mathbf{v}=\mathbf{u}+2\mathbf{v}$, 
the rack in example \ref{example3}.  The group of $2$-degenerate 2-chains $C^{ND}_2(X)$ is generated by
\[(3,1)+(1,3), \quad 2(2,2), \quad \mathrm{and} \quad 2(4,4).\]
\end{lemma}
Now since $\mathbf{u}\tr\mathbf{v}=\mathbf{u}+2\mathbf{v}$ and $\partial_2(i,j)=(i)-(i \tr j)$, the image group, $\mathrm{Im}(\partial_2)$, is two-dimensional, 
generated by two generators $(1)-(3)$ and $(2)-(4)$.  This implies the following 
\begin{lemma}
The groups $H^{R}_1(X)$ and $H^{R/ND}_1(X)$ are two-dimensional.
\end{lemma}

\begin{lemma}
  The group $\mathrm{Ker}(\partial _2)$ is 
14-dimensional.
\end{lemma}

\begin{proof}
A straightforward computation gives the following fourteen generators  $(i,2)$, $(i,4)$, $(i,1)+(i+2,3)$ for $i=1,2,3,4$ and $(i,3)+(i+2,3)$ for $i=1,2,$ of the group $\mathrm{Ker}(\partial _2)$.
\end{proof}

Now we calculate the dimension of the second homology group.

\begin{lemma}
  The group $\mathrm{Im}(\partial _3)$ is 
11-dimensional, giving dim $H_2(X)=3$.
\end{lemma}
\begin{proof}
Again a straightforward computation gives the following eleven generators $$(i,3)-(i,1), \; (i,2)-(i+2,4), \;\;\text{for}\; i=1,2,3,4 \;\;\text{and} \;(i,2)-(i+2,2) \;\;\text{for} \;i=1,2, 3.$$
\end{proof}

\section{\large\textbf{Enhancing the rack counting invariant}}

In this section we use cocycles in $H_{R/ND}^2(X;G)$ to define an 
enhancement of the rack counting invariant.

Let $L$ be a link diagram with $c$ components, $X$ a finite rack with rack rank $N$, 
$W=(\mathbb{Z}_N)^c$ and $\phi\in H_{R/ND}^2(X;G)$ a 
nondegenerate 2-cocycle with coefficients in an abelian group $G$. For a 
given coloring $f\in\mathrm{Hom}(FR(L),X)$, we define the \textit{Boltzmann 
weight} of $f$ in the following way. At each positive crossing in $L$, we 
have a contribution of $\phi(x,y)$ where $x$ is the color on the inbound 
underarc and $y$ is the color on the overarc; at each negative crossing, we
have a contribution of $-\phi(x,y)$ where $x$ is the color on the outbound
underarc and $y$ is the color on the overarc. The Boltzmann weight $BW(f)$
is then the sum over all crossings in $L$ of these contributions.
\[+\phi(x,y) \quad 
\raisebox{-0.5in}{\includegraphics{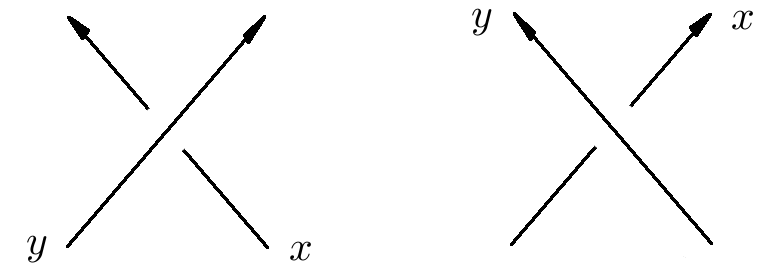}} \quad -\phi(x,y)\]
\noindent
The rack cocycle condition is precisely the condition required so that the
Boltzmann weight of a rack-colored diagram is unchanged by Reidemeister III
moves; the coloring condition is chosen to guarantee invariance under 
Reidemeister II and blackboard-framed type I moves.

\[\begin{array}{r}
+\phi(y,z)\\ 
 \\
+\phi(x,z) \\
 \\
+\phi(x\tr z,y\tr z)\\
\end{array}\raisebox{-0.8in}{
\includegraphics{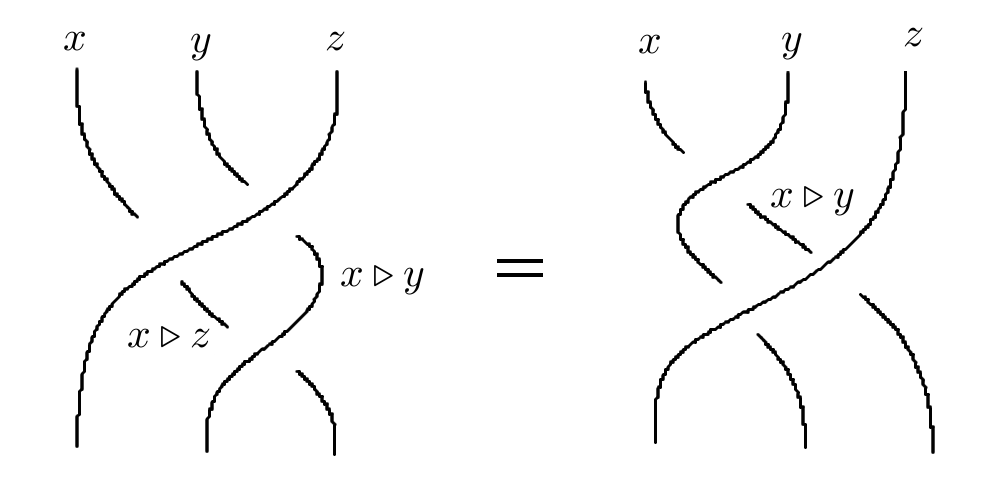}}
\begin{array}{l}
+\phi(x,y) \\
\\
+\phi(x\tr y, z) \\
\\
+\phi(y,z) \\
\end{array}
\]

\[\includegraphics{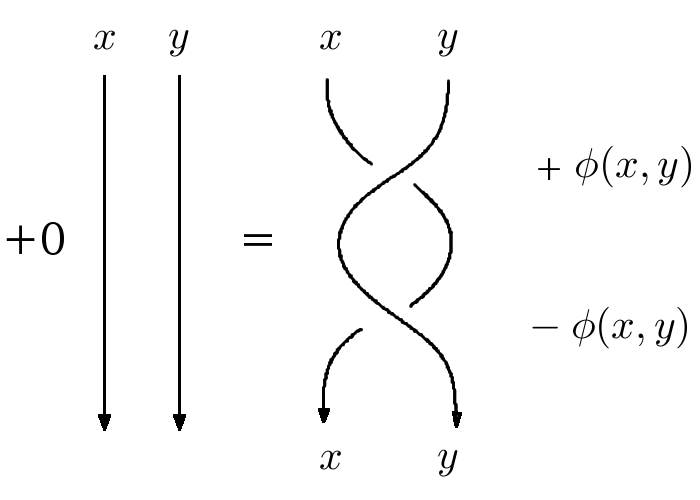} \quad \quad
\includegraphics{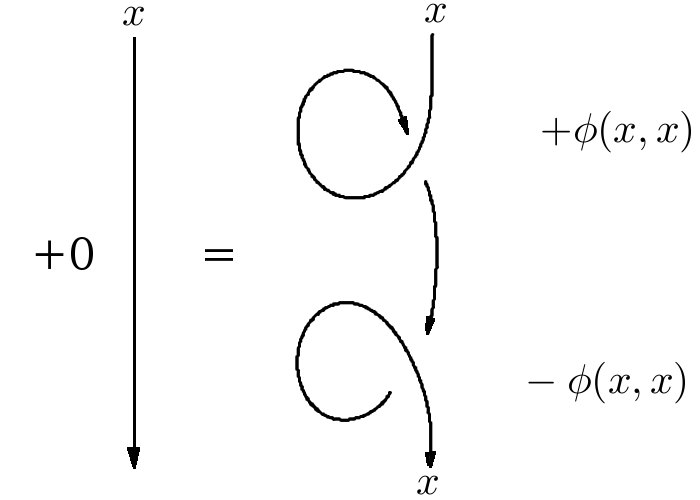}\]
\noindent
The $N$-degenerate subcomplex is generated by the characteristic chains of
$N$-phone cord tangles. By setting these equal to zero, i.e. selecting
$\phi\in H^{2}_{R/ND}(X;G)$ as opposed to $\phi\in H^{2}_{R}(X;G)$, we obtain
invariance under the $N$-phone cord move. Thus, we have

\begin{theorem}
For any cocycle $\phi\in H^2_{R/ND}(X;G)$, the multiset
\[\Phi^{M}_{\phi}(L)=\{BW(f) \ | \ f\in \mathrm{Hom}(FR(L,\mathbf{w}),X), \ 
\mathbf{w}\in W\}\]
is an invariant of ambient isotopy.
\end{theorem}

We can also use generating function style notation to rewrite the multiset
as a one-variable polynomial with $\mathbb{Z}$ coefficients and $G$ exponents:
\[\Phi_{\phi}(L)=\sum_{\mathbf{w}\in W} 
\left(\sum_{f\in \mathrm{Hom}(FR(L,\mathbf{w}),X)}u^{BW(f)}\right).\]

\begin{remark}
\textup{The invariant $\Phi^M_{\phi}(L)$ is the specialiation of the invariant
$\Phi_{\phi}(L,T)$ defined in \cite{N1} (Definition 8) obtained by setting 
$q_k=1$ for $k=1,\dots,c$.}
\end{remark}

\begin{remark}\textup{
If $N=1$ so that $X$ is a quandle, then $\Phi_{\phi}$ is the CJKLS invariant
defined in \cite{CJKLS}.
}\end{remark}

Note that there are four oriented $N$-phone cord moves; we have used only
one to define our degenerate subcomplex. One of the other moves yields the
same degenerate subcomplex, but the other two yield a slightly different
subcomplex, generated by chains of the form
\[\sum_{k=1}^N 
(x_1,\dots, x_{i-1},x_i^{\tr k}, x_i^{\tr k}, x_{i+2}, \dots, x_n).\] 
Note that these chains are called $N$-\textit{reduced} in \cite{N1}.
The phone cord moves which yield this degeneracy condition are related to the
one we chose through a combination of Reidemeister II and III moves together
with blackboard-framed I moves, so it follows that such alternative degenerate
cocycles are homologous to cocycles in $C^R_{\ast}$. In particular, we
have:

\begin{lemma}\label{phixx}
\textup{If $\phi\in H^{2}_R(X)$ then $\phi(x,x)=\phi(x,x\tr x)$.}
\end{lemma}

\begin{proof}
First note that since $\phi$ is a 2-cocycle we have
\[\phi(a,c)-\phi(a\tr b,c)-\phi(a,b)+\phi(a\tr c,b\tr c) =0\]
for all $a,b,c\in X$. Then in particular, setting $a=x\tr^{-1}x$, $b=x$ and 
$c=x$ we have
\begin{eqnarray*}
0 & = & \phi(x\tr^{-1}x,x)-\phi((x\tr^{-1} x)\tr x,x)-\phi(x\tr^{-1}x,x)
+\phi((x\tr^{-1}x)\tr x,x\tr x) \\
& = & \phi(x\tr^{-1}x,x)-\phi(x,x)-\phi(x\tr^{-1}x,x)
+\phi(x,x\tr x) \\
& = &-\phi(x,x)+\phi(x,x\tr x).
\end{eqnarray*}
\end{proof}

\begin{corollary}
Let $X$ be a rack of rack rank $1\le N<\infty$ and let
$C^{ND'}_n(X)$ be the set of linear combinations of chains of the form
\[\sum_{k=1}^N 
(x_1,\dots, x_{i-1},x_i^{\tr k}, x_i^{\tr k}, x_{i+2}, \dots, x_n).\]
Then $C_*^{ND'}(X)$ is also a subcomplex of $C^R_*(X)$ and the quotient
complexes $C^{R/ND}_*(X)$ and $C^{R/ND'}_*(X)$ are isomorphic.
\end{corollary}

\begin{remark}
\textup{It is noted in \cite{CJKLS} that 2-coboundaries have a Boltzmann 
weight contribution of zero, and hence cohomologous cocycles define the 
same invariant. We note that the exactly the same proof applies in this 
more general setting, and we have}
\[\phi_1 \mathrm{\ cohomologous \ to\ }\phi_2 \quad
\Rightarrow\quad \Phi_{\phi_1}=\Phi_{\phi_2}.\]
\end{remark}

\section{\large\textbf{Computations and examples}}

In this section we collect a few examples of the rack cocycle invariants. 
These examples were computed using \texttt{python} code available at
\texttt{www.esotericka.org} with signed Gauss codes transcribed by hand
from diagrams at the Knot Atlas \cite{KA} and checked with \texttt{Maple}.

\begin{example}\textup{
Let $X$ be the rack with rack matrix given by
\[M_{X}=\left[\begin{array}{ccccc}
1 & 3 & 2 & 1 & 1 \\
3 & 2 & 1 & 2 & 2 \\ 
2 & 1 & 3 & 3 & 3 \\ 
4 & 4 & 4 & 5 & 5 \\ 
5 & 5 & 5 & 4 & 4 \\
\end{array}\right].\]
Via computations in \texttt{python}, we selected at random a reduced 
2-cocycle $\phi\in H^2_{R/ND}(X;\mathbb{Z}_4)$ given by
\begin{eqnarray*}
\phi & = & \left(\chi_{1,3}+\chi_{3,2}+\chi_{5,4}+\chi_{5,5}\right) 
 +2\left(\chi_{1,1}+\chi_{2,2}+\chi_{3,3}\right) \\
& & \quad +3\left(\chi_{1,2}+\chi_{1,4}+\chi_{1,5}+\chi_{2,3}+\chi_{2,4}+\chi_{2,5}+\chi_{3,4}+\chi_{3,5}+\chi_{4,4}+\chi_{4,5}\right)\end{eqnarray*}
and computed $\Phi_{\phi}(L)$ for all prime classical knots with up to eight 
crossings and prime classical links with up to seven crossings. The results
are collected in the table below; note that multiple entries with equal
rack counting invariant are distinguished by the cocycle enhancement.}

\[\begin{array}{r|l}
\Phi_{\phi}(L) & L \\ \hline
5 + 3u^2 & 4_1, 5_1, 5_2, 6_2, 6_3, 7_1, 7_2, 7_3, 7_5, 7_6, 8_1, 8_2, 8_3, 8_4,
8_6, 8_7, 8_8, 8_9, 8_{12}, 8_{13}, 8_{14}, 8_{16}, 8_{17} \\
11+9u^2 & 3_1, 6_1, 7_4, 7_7, 8_5, 8_{10}, 8_{11}, 8_{15}, 8_{19}, 8_{20}, 8_{21} \\
29+27u^2 & 8_{18} \\
10 + 12u+ 6u^2 + 12u^3 & L2a1, L6a2, L7a6 \\
22 + 12u + 18u^2 + 12u^3 & L6a3, L7a5 \\
22+18u^2 & L4a1, L5a1, L7a4 \\
34+30u^2 & L6a1, L7a1, L7a2, L7a3, L7n1, L7n2 \\
92+84u^2 & L6n1, L7a7 \\
116 + 108u^2 & L6a5 \\
164 + 156u^2 & L6a4 \\
\end{array}\]

\end{example}

\begin{example}\textup{
Let $X$ be the rack with rack matrix given by
\[M_{X}=\left[\begin{array}{cccccc}
2 & 2 & 1 & 1 & 1 & 1 \\
1 & 1 & 2 & 2 & 2 & 2 \\
3 & 3 & 3 & 5 & 6 & 4 \\
4 & 4 & 6 & 4 & 3 & 5 \\ 
5 & 5 & 4 & 6 & 5 & 3 \\
6 & 6 & 5 & 3 & 4 & 6
\end{array}\right]\]
and let $\phi$ be the reduced
2-cocycle $\phi\in H^2_{R/ND}(X;\mathbb{Z}_4)$ given by}
\begin{eqnarray*}
\phi & = & \chi_{1,3}+\chi_{1,4}+\chi_{1,5}+\chi_{1,6}+\chi_{2,3}+\chi_{2,4}+\chi_{2,5}+\chi_{2,6}+\chi_{3,1}+\chi_{3,2}+\chi_{3,4}+\chi_{3,5} +\chi_{4,1} +\chi_{4,2}
 \\ & &
+\chi_{4,3}+\chi_{5,1}+\chi_{5,2}+\chi_{5,6}+\chi_{6,1}+\chi_{6,2}+\chi_{6,4}+\chi_{6,5} +2\left(\chi_{4,5}+\chi_{5,4}\right)
 +3\left(\chi_{4,6}+\chi_{5,3}\right);
\end{eqnarray*}
\textup{The invariant values are listed in the table below. Unlike the previous 
example, this example includes single-component knots which are distinguished
by the cocycle enhancement but not the counting invariant.}

\[\begin{array}{r|l}
\Phi_{\phi}(L) & L \\ \hline
10 & 5_1,5_2,6_1,6_2,6_3,7_1,7_4,7_5,7_6,7_7,8_2,8_3,8_6,8_7,8_8,8_9,8_{10},8_{12},
8_{14},8_{16},8_{17} \\
10+24u^2 & 3_1, 4_1,7_2,7_3,8_1,8_4,8_{11},8_{13},8_{18} \\
34 & 8_5,8_{15},8_{19},8_{20},8_{21} \\
34+96u^2 & 8_{18} \\
52 & L5a1, L6a1 \\
100 & L4a1, L7a4 \\
232 & L6a4 \\
20+32u^2 & L2a1, L6a2, L7a5, L7a6 \\
52+48u^2 & L7a1, L7a2, L7a3, L7n1, L7n2 \\
68+32u^2 & L6a3 \\
104+128u^2 & L6a5 \\
232+96u^2 & L6n1, L7a7 \\
\end{array}\]

\end{example}

The previous examples both use racks which are disjoint unions of 
one rack and one quandle, with each orbit acting trivially on the 
other. For our final example, we use a more ``pure rack'' with no
elements of rack rank $1$.

\begin{example}
\textup{Consider the rack with rack matrix}
\[M_{X}=\left[\begin{array}{cccc}
2 & 2 & 1 & 1 \\
1 & 1 & 2 & 2 \\
4 & 4 & 4 & 4 \\
3 & 3 & 3 & 3 \\
\end{array}\right]\]
\textup{and let $\phi\in H^2_{R/ND}(X;\mathbb{Z}_6)$ be given by}
\[\phi=\chi_{2,1}+\chi_{2,2}
+3\left(\chi_{3,3}+\chi_{3,4}+\chi_{4,3}+\chi_{4,4}\right)
+5\left(\chi_{1,1}+\chi_{1,2}+\chi_{3,1}+\chi_{3,2}+\chi_{4,1}+\chi_{4,2}\right)
\]
\textup{The invariant has value $\Phi_{\phi}(K)=4$ for all of the knots
$K$ in our list; however, the invariant is nontrivial on links:}
\[\begin{array}{r|l}
\Phi_{\phi}(L) & L \\ \hline
16 & L5a1, L6a3, L7a1, L7a3, L7a4, L7n2 \\
8+8u^2 & L2a1, L7a5, L7a6, L7n1 \\
8+8u^4 & L4a1, L6a1, L7a2 \\
12+4u^4 & L6a2 \\
16+48u^4 & L6a4, L6n1 \\
48+16u^4 & L6a5, L7a7
\end{array}\]

\end{example}

\section{A Compendium of Questions}
We conclude the paper by the following remarks and open questions:

\begin{list}{$\bullet$}{}
\item{In \cite{CES1} a homology theory was developed for set-theoretic 
Yang-Baxter equations and used to define invariants of classical and 
virtual knots.  This approach was used in \cite{CN} and extended to detect 
non-classicality of some virtual links. Within these lines generalize the 
cohomology theory in this paper to biracks and use it study invariants of 
classical and virtual knots.}
\item{How can we extend these rack cocycle invariants to surface knots
and links, or more generally, knotted compact oriented $n$-manifolds 
in $S^{n+2}$?} 
\item{Recent work (e.g., \cite{PN}) has found that quandle homology groups
sometimes have additional algebraic structure. How can we use these structures 
to further enhance the rack 2-cocycle invariants?}
\item{More generally, what kinds of patterns (long exact sequecnes, etc.) 
exist in the $R/ND$ homology groups of racks of various types?}
\item{Our computer experiments suggest that small cardinality racks define
2-cocycle invariants which are stronger on links than on knots, due to the
orbit subracks being fairly trivial constant action racks. Fast algorithms
for computing the homology of large-cardinality racks should improve the
practical utility of the rack cocycle invariants.}
\item{Given any rack $X$, there are a number of quandles associated to $X$
such as the \textit{maximal subquandle} of $X$ consisting of all elements
of $X$ of rack rank $0$ (this may be the empty quandle) and the 
\textit{operator quandle} obtained by taking the quotient of $X$ by the 
congruence $\{x\sim y \iff z\tr x=z\tr y\ \forall z\in X.\}$ What 
relationship, if any, can be found between the quandle homology of these 
quandles and the $R/ND$ homology of $X$?}
\end{list}

\noindent
\textsc{Department of Mathematics, \\
University of South Florida, \\
4202 E Fowler Ave., \\
Tampa, FL 33620}\\\\
and\\\\
\textsc{Department of Mathematics, \\
Claremont McKenna College, \\
850 Columbia Ave., \\
Claremont, CA 91711}

\end{document}